\documentclass[11pt,letterpaper,reqno,abstract=true]{scrartcl}

\usepackage{amsmath}
\usepackage{amsfonts}
\usepackage{amssymb}
\usepackage{amsthm}
\usepackage{amscd}

\usepackage{tikz}
\usepackage{tikz-cd}
\usetikzlibrary{cd}

\usepackage[colorlinks=true]{hyperref}
\hypersetup{urlcolor=red, citecolor=blue}

\usepackage{latexsym}
\usepackage{mathrsfs}
\usepackage{psfrag}
\usepackage[dvips]{epsfig}
\usepackage{epsfig}
\usepackage[all]{xy}         % Commutative diagrams
\usepackage{wrapfig}

\usepackage[margin=2.5cm]{geometry}

\theoremstyle{plain}                              % Italicized
\newtheorem*{mainthm}{Main Theorem}

\newtheorem{thm}{Theorem}[section]
\newtheorem{prop}[thm]{Proposition}
\newtheorem{defn}[thm]{Definition}
\newtheorem{lem}[thm]{Lemma}
\newtheorem{cor}[thm]{Corollary}

\theoremstyle{definition}                         % Non-italicized

\newtheorem*{remark}{Remark} 

\theoremstyle{remark}                             % Non-italicized &
\newcommand{\R}{\mathbb{R}}                     % Reals
\newcommand{\veps}{\varepsilon}        

\newcommand{\del}{\partial}

\newcommand{\restrict}{\upharpoonright}

\DeclareMathOperator{\vol}{vol}

\providecommand{\norm}[1]{\left\lVert#1\right\rVert}       % Norm 

\providecommand{\abs}[1]{\left\lvert#1\right\rvert}        % Absolute value

\begin{document}

\title{The Liv\v{s}ic equation on differential forms over Anosov flows
  and applications}

\author{Slobodan N. Simi\'c}

\maketitle

\begin{abstract}
  The goal of this paper is to explore the relationship between the
  geometric properties of an Anosov flow on a closed manifold $M$ and
  the analytic properties of its infinitesimal generator $X$ as a
  linear operator on the space of smooth differential forms of all
  degrees. In particular, we study the solvability of the Liv\v{s}ic
  equation $L_X \xi = \eta$ on the space of differential forms and
  show, for instance, that if the Anosov flow is \emph{asymmetric},
  then the equation has a unique solution in the continuous category
  in degrees $2 \leq k \leq n-2$, where $n = \dim M$. Intuitively, an
  Anosov flow is asymmetric if in negative time it shrinks the volume
  of any $(n-2)$-dimensional parallelepiped exponentially fast when at
  least one side of it is in the strong unstable direction. As an
  application, we show that for volume-preserving asymmetric Anosov
  flows, the following result holds: the $L^2$-closure of the image of
  $L_X$ restricted to differential forms of degree $(n-1)$ contains
  the space of $L^2$-exact $(n-1)$-forms if and only if the sum of the
  strong bundles of the flow is uniquely integrable, in which case the
  flow is therefore topologically conjugate to a suspension of an
  Anosov diffeomorphism.
\end{abstract}

%\address{Department of Mathematics and Statistics, San Jos\'e State University,% San Jose, CA 95192-0103}

%\urladdr{http://www.math.sjsu.edu/$\tilde{\:}$simic}

%\email{slobodan.simic@sjsu.edu}

%\thanks{Partially supported by SJSU Research, Scholarship, and Creative Activity grants.}

%\subjclass{}
\date{\today}
%\dedicatory{}
%\keywords{}

%\input intro

\section{Introduction}
\label{sec:introduction}

Let $X$ be a smooth\footnote{We
  use the terms smooth and $C^\infty$ interchangeably.} vector field on a smooth closed (i.e., compact and without boundary) connected manifold $M$. It is natural to ask: how are various properties of the flow $\Phi$ generated by $X$ related to the properties of the differential operator $X$ (or $L_X$, the Lie derivative) acting on some space of functions, distributions, or differential forms? Much work has been done on this question. Some notable results are those of Liv\v{s}ic \cite{livsic+71, livsic+72} in the 1970's on the equation $X\varphi = f$, nowadays known as the \textsf{Liv\v{s}ic equation}. See Section~\ref{sec:livshitz}.

There have been numerous recent results (cf., e.g., \cite{faure+sjostrand+2011, G+L+P+2013, dyatlov+zworski+2016}; see also \cite{lefeuvre_book+2025} and the sources listed therein) relating the properties of the spectrum of the differential operator $X$ (or a related transfer operator) acting on suitable spaces to the \emph{statistical} properties of $\Phi$. The goal of this paper is to explore what properties of $X$ as a differential operator can tell us about the \emph{geometric} properties of $\Phi$. More precisely, we look at the properties of the Lie derivative $L_X$ acting on the space of differential forms, and ask the following natural question:

\begin{quote}
  \textit{Given a continuous differential form $\eta$, does the Liv\v{s}ic equation $L_X \xi = \eta$ have a continuous solution $\xi$?}
\end{quote}

We show that the answer is affirmative in intermediate degrees (i.e., $2 \leq k \leq n-2$, where $n = \dim M$) for Anosov flows we call asymmetric. In other degrees, we characterize the image of the operator $L_X$.

\begin{defn}
  We call an Anosov flow on $M$ \textsf{asymmetric} if for every $x \in M$ and every $(n-2)$-dimensional parallelepiped $\Pi$ in the tangent space $T_x M$ with least one side of $\Pi$ in the strong unstable space $E^{uu}_x$, then in negative time, the action of derivative of the flow on $\Pi$ shrinks its volume exponentially fast.
\end{defn}

That is, if $\Phi = \{ f_t \}$, then
\begin{displaymath}
  \text{vol}(T_x f_{-t}(\Pi)) \leq C e^{-\lambda t} \text{vol}(\Pi),
\end{displaymath}
for some $C, \lambda > 0$ and all $t \geq 0$. In other words, in negative time, the rate of contraction along the strong unstable direction dominates the joint rate of expansion in the remaining directions (including those in the strong stable bundle). It is clear that this notion makes sense only if $n \geq 4$.

Observe that if $\Phi$ is asymmetric, then the exponential shrinking of the volume also holds for all \emph{lower dimensional} parallelepipeds with at least one side in the strong unstable bundle.

The structure of the paper is the following. In Section~\ref{sec:preliminaries} we review some basic facts about Anosov flows and the space of differential forms as an inner product space. In Section~\ref{sec:livshitz} we prove the Liv\v{s}ic theorem for differential forms in intermediate degrees and related properties of the Lie derivative in other degrees. Our main result is an application of these properties; the proof is given in Section~\ref{sec:proof}.

\begin{mainthm}
  Let $\Phi$ be an asymmetric Anosov flow with infinitesimal generator $X$ on a closed Riemannian manifold $M$ of dimension $n \geq 4$. Then: the $L^2$-closure of the image of $L_X$ on $(n-1)$-forms\footnote{We consider the domain of $L_X$ to be the space of continuous differential forms with a continuous $L_X$-derivative.} contains the space of $L^2$-exact forms, i.e.,
  \begin{displaymath}
    L^2 B^{n-1}(M) \subset \overline{\text{\normalfont image}(L_X)},
  \end{displaymath}
  if and only if the sum of the strong bundles of $\Phi$ is uniquely integrable and the flow is therefore topologically conjugate to a suspension of an Anosov diffeomorphism.
\end{mainthm}

The closure is taken relative to the $L^2$-norm (see Section~\ref{sec:preliminaries}).

\section{Preliminaries}
\label{sec:preliminaries}

\paragraph{Anosov flows.} Fix a non-singular smooth flow
$\Phi = \{ f_t \}$ on a closed Riemannian manifold $M$. Recall that
$\Phi$ is called \textsf{Anosov} if there exists a $Tf_t$-invariant
splitting of the tangent bundle into the strong unstable, center, and
strong stable bundle,
\begin{displaymath}
  TM = E^{uu} \oplus E^c \oplus E^{ss},
\end{displaymath}
such that for all $t \geq 0$, $v \in E^{ss}$ and $w \in E^{uu}$, we
have:
\begin{displaymath}     \tag{$\spadesuit$}
\norm{Tf_t (v)} \leq c e^{-\nu t } \norm{v}
\qquad \qquad \text{and} \qquad
\qquad \norm{Tf_t(w)} \geq c e^{\lambda t} \norm{w},
\end{displaymath}
where $c, \nu$, and $\lambda$ are fixed positive constants, and $E^c$
is spanned by the infinitesimal generator $X$ of the flow. The Anosov
property is independent of the Riemannian metric, since on a compact
manifold the Finsler structures defined by any two continuous
Riemannian metrics are equivalent.

An Anosov flow is of \textsf{codimension one} if $\dim E^{uu} = 1$ or
$\dim E^{ss} = 1$. We will always assume the former. It is
\textsf{volume-preserving} if there exists a $C^\infty$ volume form
$\Omega$ such that $f_t^\ast \Omega = \Omega$, for all $t \in \R$.

It is well-known that the invariant bundles $E^{ss}$, $E^{uu}$,
$E^{cs} = E^c \oplus E^{ss}$, and $E^{cu} = E^c \oplus E^{uu}$ are
uniquely integrable, giving rise to H\"older continuous invariant
foliations \cite{hps77,psw+97} denoted by $W^{cs}, W^{cu}, W^{ss}$,
$W^{uu}$, respectively.

A smooth compact codimension one submanifold $\Sigma$ of $M$ is called
a \textsf{global cross section} for a flow if it intersects every
orbit transversely. If a flow admits a global cross section $\Sigma$,
then every point $p \in \Sigma$ returns to $\Sigma$, defining the
Poincar\'e or first-return map $g: \Sigma \to \Sigma$ of the flow. The
flow can be reconstructed by \textsf{suspending} $g$ under the roof
function equal to the first-return time (cf., e.g., \cite{katok+95}).

\paragraph{Notation, standing assumptions, and facts.}

Below we recall some basic facts, and fix the notation and terminology
used in this paper.

\begin{enumerate}

\item $\Phi = \{ f_t \}$ denotes a $C^\infty$ Anosov flow on a closed
  connected $C^\infty$ Riemannian manifold $M$ of dimension
  $n$.

\item $X$ denotes the associated infinitesimal generator of $\Phi$;
  $L_X$ is the corresponding Lie derivative on tensor fields. The
  restriction of $L_X$ to differential forms of degree $k$ will be
  denoted by $L_X^{(k)}$. When there is little chance of confusion,
  the superscript $k$ will be dropped.

\item $C^1_X$ will denote any space of continuous objects whose
  $L_X$-derivative is continuous. Thus $C^1_X(M)$ is the space of
  continuous function with a continuous $X$-derivative.
  $C^1_X \Lambda^k(M)$ will denote the space of continuous $k$-forms
  $\omega$ such that $L_X \omega$ is continuous.

\item $\Omega$ is a $C^\infty$ volume form invariant under the
  flow. Without loss we assume that $\int_M \Omega = 1$.

\item A standing assumption is that all invariant bundles are
  orientable; otherwise we can pass to a double cover of
  $M$. 

\item $E^{su} = E^{ss} \oplus E^{uu}$; $E^{su}$ is a H\"older
  continuous bundle (cf., \cite{hassel+94, hassel+97, hps77}).

\item If the flow is of codimension one and $n \geq 4$, $E^{uu}$ and
  $E^{cs}$ are both known to be $C^1$ (in fact, $C^{1 + \theta}$, for
  some $0 < \theta < 1$); cf., \cite{hassel+94, hassel+97, hps77}.

\item We denote by $\alpha$ the canonical invariant 1-form defined by:
  \begin{displaymath}
    \ker(\alpha) = E^{su}, \qquad \alpha(X) = 1.
  \end{displaymath}
  The regularity of $\alpha$ is the same as that of $E^{su}$, i.e.,
  H\"older continuous.

  \item We will call a continuous Riemannian metric $g$ on $M$ an
    \textsf{Anosov metric} associated with a fixed Anosov flow $\Phi$
    if relative to $g$, $X$ is orthogonal to $E^{su}$ and
    $g(X,X) = 1$.

  \item For a Riemannian metric $g$, its Riemannian volume form is
    denoted by $\vol(g)$.
    
  \item For an arbitrary continuous Riemannian metric $g$ with
    $\vol(g) = \Omega$ and vectors $v_1, \ldots, v_k$ tangent to $M$
    at the same point, we write
    $\norm{v_1 \wedge \cdots \wedge v_k}_g$ for the $k$-dimensional
    volume with respect to $g$ of the parallelepiped with sides
    $v_1, \ldots, v_k$. Thus
    $\norm{v_1 \wedge \cdots \wedge v_n}_g = \abs{\Omega(v_1,\ldots,
      v_n)}$.

  \item Recall that a continuous 1-form $\omega$ on $M$ is said to
    have an exterior differential in the \textsf{Stokes sense} if
    there exists a continuous 2-form $\xi$ such that
    \begin{displaymath}
      \int_{\del D} \omega = \int_D \xi,
    \end{displaymath}
    for every $C^1$-immersed 2-disk $D$ such that $\del D$ is
    piecewise $C^1$. In that case we write $\xi = d\omega$, specifying
    that this holds in the Stokes sense. The Hartman-Frobenius theorem
    (i.e., P. Hartman's generalization of the classical theorem of
    Frobenius on integrability of plane fields; see \cite{plante72}
    and \cite{hart02}) states that a continuous 1-form $\omega$ is
    integrable if and only if $\omega$ has a continuous exterior
    differential $d \omega$ in the Stokes sense and
    $\omega \wedge d\omega = 0$. Recall that a continuous 1-form
    $\omega$ on $M$ is said to be integrable if the kernel of $\omega$
    as a subbundle of $TM$ is integrable.

\end{enumerate}

\paragraph{The case of codimension one Anosov flows.} We will show
that codimension one Anosov flows in dimensions $n \geq 4$ are asymmetric.

\begin{prop} \label{prop:exp-contr}
  Let $\Phi = \{ f_t \}$ be a volume preserving codimension one Anosov
  flow on a closed manifold $M$ of dimension $n \geq 4$. Assume, without
  loss, that $E^{uu}$ is 1-dimensional and orientable, and let $Y$ be a
  non-vanishing section of $E^{uu}$. Assume as before that with
  respect to a fixed Riemannian metric $g$ and the associated Finsler
  structure on $M$:
  \begin{displaymath}
    \norm{T f_t(v)} \leq c e^{-\nu t} \norm{v},
  \end{displaymath}
  for all $t \geq 0$ and $v \in E^{ss}$, where $c, \nu > 0$ are as in
  $(\spadesuit)$. Then for every $p \in M$, all linearly independent
  unit vectors $v_1, \ldots, v_{n-3} \in T_p M$, and $t > 0$, we have
  \begin{displaymath}
    \norm{T_p f_{-t} \left( v_1 \wedge \cdots \wedge v_{n-3}
      \wedge Y_p \right)}_g \leq C e^{-\nu t},
  \end{displaymath}
  where $C > 0$ is independent of $p$,
$v_i$'s, and $t$.
\end{prop}

\begin{proof}
  We first assume that $g$ is an Anosov metric for $\Phi$. Let us deal
  with the worst-case scenario, i.e., when $v_1, \ldots, v_{n-3}$ are
  all in $E^{ss}_p$.

  Fix $t > 0$. Let $w_t \in E^{ss}_p$ be a unit vector such that
  $T_p f_{-t}(w_t)$ is orthogonal to the subspace spanned by
  $Y_{f_{-t}(p)}$ and $T_p f_{-t}(v_i)$, for $i = 1,\ldots, n-3$, and
  $(v_1, \cdots, v_{n-3}, Y_p, w_t, X_p)$ is a positively oriented
  basis of $T_p M$. Note that
  $\norm{T f_{-t}(w_t)} \geq c^{-1} e^{\nu t}$. Since $\Phi$ leaves
  $\Omega$ invariant, we have:
  \begin{align*}
    \norm{T_p f_{-t} \left( v_1 \wedge \cdots \wedge v_{n-3} 
        \wedge Y_p \wedge w_t \wedge X_p \right)}_g 
    & = (f_{-t}^\ast \Omega)(v_1, \ldots, v_{n-3}, Y_p, w_t, X_p) \\
    & = \Omega(v_1, \ldots, v_{n-3}, Y_p, w_t, X_p) \\
    & = \norm{v_1 \wedge \cdots \wedge v_{n-3} \wedge Y_p \wedge w_t
     \wedge X_p}_g \\
   & \leq \norm{Y}_\infty,
  \end{align*}
  where $\norm{Y}_\infty = \max \{ \norm{Y_x} : x \in M \}$; here we
  used $g(X_p,X_p) = 1$, for all $p \in M$. Our choice of $w_t$
  implies
  \begin{align*}
    \norm{T_p f_{-t} \left( v_1 \wedge \cdots \wedge v_{n-3} 
            \wedge Y_p \wedge w_t \wedge X_p \right)}_g
     & = \norm{T_p f_{-t} \left( v_1 \wedge \cdots \wedge v_{n-3} 
            \wedge Y_p \wedge w_t \right) \wedge X_{f_{-t}p}}_g \\
     & = \norm{T_p f_{-t} \left( v_1 \wedge \cdots \wedge v_{n-3} 
           \wedge w_t \wedge Y_p \right)}_g \norm{X_{f_{-t}(p)}}  \\
     & =  \norm{T_p f_{-t} \left( v_1 \wedge 
         \cdots \wedge v_{n-3} \wedge Y_p \right)}_g \norm{T_p
       f_{-t}(w_t)} \\
    & \geq c^{-1} e^{\nu t} \norm{T_p f_{-t} \left( v_1 \wedge 
         \cdots \wedge v_{n-3} \wedge Y_p \right)}_g,
  \end{align*}
  since $T_p f_{-t}(w_t)$ is orthogonal to the subspace containing the
  parallelepiped
  $T_p f_{-t} \left( v_1 \wedge \cdots \wedge v_{n-3} \wedge Y_p
  \right)$. Combining the last two inequalities, we obtain for all
  $t \geq 0$:
  \begin{displaymath}
    \norm{T_p f_{-t} \left( v_1 \wedge \cdots \wedge v_{n-3}
    \wedge Y_p \right)}_g \leq c \norm{Y}_\infty e^{-\nu t}.
  \end{displaymath}
  If $g$ is not Anosov, then the Finsler norms defined by $g$ and any
  fixed Anosov metric $g_0$ are equivalent, yielding an analogous
  inequality.
\end{proof}

Therefore, we have:

\begin{cor}
  Volume-preserving codimension one Anosov flows in dimensions $n \geq
  4$
  are asymmetric.
\end{cor}

\paragraph{The $L^2$-structure on the space of differential forms.}

For $0 \leq r \leq \infty$ and $0 \leq k \leq n$, $C^r \Lambda^k(M)$
will denote the space of $C^r$ exterior differential forms of degree
$k$ on $M$. The space of exact $C^r$ $k$-forms will be denoted by
$C^r B^k(M)$ and the space of closed $C^r$ $k$-forms by $C^r Z^k(M)$.

On any oriented inner product space $V$ with inner product
$\langle \cdot, \cdot \rangle$ and corresponding volume form $\omega$,
one can uniquely define an inner product on exterior forms so that, in
particular, $\langle \xi, \eta \rangle = \langle u, v \rangle$, for
all exterior 1-forms $\xi, \eta$, where $u, v$ are the vectors dual to
$\xi, \eta$, respectively, relative to $\langle \cdot, \cdot
\rangle$. See \cite{lee+smooth+13,warner+83}.

Recall also that for each $0 \leq k \leq n$ (where $n = \dim V$) there
is a unique isomorphism
$\star : \Lambda^k(V^\ast) \to \Lambda^{n-k}(V^\ast)$, called the
\textsf{Hodge-star operator}, between the spaces of exterior $k$- and
$(n-k)$-forms on $V$ such that
\begin{equation}   \label{eq:Hodge-star}
  \xi \wedge \star \eta = \langle \xi, \eta \rangle \: \omega,
\end{equation}
for any $\xi, \eta \in \Lambda^k(V^\ast)$.

The following lemma will be needed later in the paper. The proof is
elementary (it is an exercise in \cite{lee+smooth+13}) and therefore
omitted.

\begin{lem}       \label{lem:star} 
  If $V, \langle \cdot, \cdot \rangle$, and $\omega$ are as above,
  then for every $v \in V$, we have
  \begin{displaymath}
    \star (i_v \omega) = (-1)^{n-1} \theta_v,
  \end{displaymath}
  where $\theta_v = \langle v, \cdot \rangle$ is the exterior 1-form
  dual to $v$ relative to $\langle \cdot, \cdot \rangle$.
  \end{lem}

  If $g$ is a $C^r$ ($0 \leq r \leq \infty$) Riemannian metric on $M$
  with $\vol(g) = \Omega$, we will denote by $\star_g$ the associated
  Hodge-star operator \cite{lee+smooth+13,warner+83}, defined
  pointwise as in \eqref{eq:Hodge-star}.  The inner product on
  $C^0 \Lambda^k(M)$ (for each $0 \leq k \leq n$) induced by $g$
  is defined by
  \begin{displaymath}
    \langle \xi,\eta \rangle_g = \int_M \xi \wedge \star_g \eta.
  \end{displaymath}
  The metric $g$ defines a Finsler structure on $M$, which we denote
  by $\abs{\cdot}_g = \langle \cdot, \cdot \rangle_g^{1/2}$. For any
  form $\omega \in C^0 \Lambda^k(M)$, with $1 \leq k \leq n$ and
  $p \in M$, we will denote by $\abs{\omega_p}_g$ the operator norm of
  $\omega_p : (T_x M)^k \to \R$ as a $k$-linear map relative to this
  Finsler structure:
\begin{displaymath}
  \abs{\omega_p}_g = \max \{ \abs{\omega_p(u_1,\ldots,u_k)}_g : u_j \in T_x M,
  \abs{u_j}_g = 1 \}.
\end{displaymath}
The $C^0$-norm of $\omega$ is defined by
\begin{displaymath}
  \norm{\omega}_\infty = \sup \{ \abs{\omega_p}_g : p \in M \}.
\end{displaymath}
This is to be distinguished from the $L^2$-norm
$\norm{\omega}_g = \langle \omega, \omega \rangle_g^{1/2}$. The
completion of $C^0 \Lambda^k(M)$ relative to this norm is the space
$L^2 \Lambda^k(M)$.

In an analogous way we can define an inner product on the space of
continuous vector fields on $M$ by setting
\begin{displaymath}
  \langle Y, Z \rangle_g = \int_M g(Y,Z) \Omega.
\end{displaymath}
The corresponding $L^2$-norm is denoted by $\norm{Z}_g = \langle
Z, Z \rangle_g^{1/2}$.

As a direct consequence of Lemma~\ref{lem:star}, we have:

    \begin{cor}  \label{cor:dual_form}
    If $\Omega$ is a volume form on $M$, $Z$ a non-vanishing vector
    field, and $g$ a Riemannian metric with $\vol(g) = \Omega$, then 
    \begin{displaymath}
      \star_g (i_Z \Omega) = (-1)^{n-1} \theta_Z,
    \end{displaymath}
    where $\theta_Z = g(Z,\cdot)$. If $X, \Omega$, and $\alpha$ are
    defined as before, and $g$ is an Anosov metric for the flow, then
    \begin{displaymath}
      \star_g (i_X \Omega) = (-1)^{n-1} \alpha.
    \end{displaymath}
  \end{cor}

  Consider now the unbounded linear operator
  \begin{displaymath}
    L_X^{(k)} : L^2 \Lambda^k(M) \to L^2 \Lambda^k(M),
  \end{displaymath}
  with dense domain $C^1_X \Lambda^k(M)$. The underlying Riemannian
  metric $g$ (used to define the $L^2$-inner product on the space of
  differential forms) is assumed to be at least of class $C^1_X$
  (i.e., continuous with a continuous $L_X$-derivative); note that
  this includes Anosov metrics. We have:

  \begin{prop}     \label{prop:general-L_X}
    (a) The  adjoint of $L_X^{(k)}$ is
  \begin{displaymath}
    \left( L_X^{(k)} \right)^\ast = (-1)^{k(n-k) + 1} \star_g L_X^{(n-k)} \star_g.
  \end{displaymath}
  (b) $L_X^{(k)}$ is a closed operator and
  $\left( L_X^{(k)} \right)^{\ast \ast} = L_X^{(k)}$,  for all $0 \leq k \leq n$.

    \noindent (c) We have
    \begin{displaymath}
      [\text{\normalfont image}(L_X^{(k)})]^{\perp_g} =
      \ker(\star_g L_X^{(n-k)} \star_g) \quad
      \text{and} \quad [\text{\normalfont image}(\star_g L_X^{(n-k)}
      \star_g)]^{\perp_g} = \ker(L_X^{(k)}).
    \end{displaymath}
  \end{prop}

  Here $S^{\perp_g}$ denotes the orthogonal complement of a set $S$
  relative to the $L^2$-inner product defined by the Riemannian metric $g$.
  
  \begin{proof}
    (a) Let $\xi, \eta \in C^1_X \Lambda^k(M)$. Then:
    \begin{align*}
      \langle L_X^{(k)} \xi, \eta \rangle_g & = \int_M L_X^{(k)} \xi
                                              \wedge \star_g \eta \\
                                            & = - \int_M \xi \wedge L_X^{(n-k)} \star_g \eta \\
      & = (-1)^{k(n-k)+1} \int_M \xi \wedge \star_g [\star_g
        L_X^{(n-k)} \star_g \eta] \\
      & = (-1)^{k(n-k)+1} \langle \xi, \star_g L_X^{(n-k)} \star_g
        \eta \rangle_g,
    \end{align*}
    which proves (a). We used the fact that on $k$-forms, $\star_g
    \star_g = (-1)^{k(n-k)} \text{id}$. \\
    
    \noindent (b) Recall (see, e.g, \cite{conway_fa+2007}) that a
    densely defined unbounded operator is closable if its adjoint is
    densely defined. Since $\star_g$ maps $C^1_X$-forms to
    $C^1_X$-forms, the domain of $\star_g L_X^{(n-k)} \star_g$ is
    $C^1_X \Lambda^k(M)$, which is dense in $L^2 \Lambda^k(M)$, so
    $L_X^{(k)}$ is closable. To compute its second adjoint, we have:
    \begin{align*}
      \langle \left( L_X^{(k)}\right)^\ast \xi, \eta \rangle_g & = \langle
                                                            \eta,
                                                            \left(
                                                            L_X^{(k)}\right)^\ast
                                                            \xi
                                                                 \rangle_g \\
      & = \int_M \eta \wedge \star_g \left( L_X^{(k)}\right)^\ast \xi \\
                                                          & = - \int_M \eta \wedge L_X^{(n-k)} \star_g \xi \\
                                                          & = \int_M L_X^{(k)} \eta \wedge \star_g \xi \\
                                                          & = \langle L_X^{(k)} \eta, \xi \rangle_g \\
      & = \langle \xi, L_X^{(k)} \eta \rangle_g.
    \end{align*}
    Thus $\left( L_X^{(k)} \right)^{\ast \ast} = L_X^{(k)}$, so
    $L_X^{(k)}$ is in fact closed, being the adjoint of another
    operator (see \cite{conway_fa+2007}).

    \medskip
    
    (c) A direct consequence of the general theory of unbounded linear
    operators (see \cite{conway_fa+2007}, Proposition X.1.13) and (a).
  \end{proof}

\paragraph{The Gol'dshtein-Troyanov complex}
To make the paper as self-contained as possible, we briefly review a
result from \cite{goldshtein+troyanov+06} we will need later. In
\cite{goldshtein+troyanov+06}, Gol'dshtein and Troyanov define the
following spaces:
\begin{displaymath}
  \Omega^k_{p,q}(M) = \{ \omega \in L^q \Lambda^k(M) : d\omega \in L^p
  \Lambda^{k+1}(M) \},
\end{displaymath}
where $(M,g)$ is a Riemannian manifold (which we assume to be
compact), $1 \leq p, q \leq \infty$, and $d$ denotes the \emph{weak} exterior
differential. For each $p$ and $q$, this is a Banach space with the
graph norm
\begin{displaymath}
  \norm{\omega}_{\Omega_{p,q}} = \norm{\omega}_{L^q} + \norm{d\omega}_{L^p}.
\end{displaymath}
The spaces $\Omega^k_{p,q}(M)$ are used to define the so called
$L_{p,q}$-cohomology of $M$, which we do not need here. We will
however use some parts of the following result (Theorem 12.5 in
\cite{goldshtein+troyanov+06}):

\begin{thm}[The regularization and homotopy operators]    \label{thm:goldshtein+troyanov}
  There exists a family of regularization operators $R_\veps$ and
  homotopy operators $A_\veps$ (with $\veps > 0$) satisfying the following properties:
  
  \begin{enumerate}
  \item[(a)] For every $\omega \in L^1 \Lambda^k(M)$, the form
    $R_\veps \: \omega$ is smooth.
  \item[(b)] For any $\omega \in \Omega^k_{q,p}(M)$, we have $d R_\veps
    \: \omega = R_\veps \: d\omega$.
  \item[(c)] For any $1 \leq p, q < \infty$ and $\veps > 0$, $R_\veps
    : \Omega^k_{q,p}(M) \to \Omega^k_{q,p}(M)$ is a bounded linear operator
    such that $\norm{R_\veps}_{q,p} \to 1$, as $\veps \to 0$.
  \item[(d)] For any $1 \leq p, q < \infty$ and $\omega \in
    \Omega^k_{q,p}(M)$, we have $\norm{R_\veps \omega - \omega}_p \to
    0$, as $\veps \to 0$. Thus smooth forms are dense in
    $\Omega^k_{q,p}(M)$ (if $p, q$ are finite).
  \item[(e)] The homotopy operator $A_\veps : \Omega^k_{p,r}(M) \to
    \Omega^{k-1}_{q,p}(M)$ (where $1 \leq k \leq n$) is bounded in the
    following cases:
    \begin{enumerate}
    \item[(i)] If $1 \leq p, q, r \leq \infty$ satisfy $\frac{1}{p} -
      \frac{1}{q} < \frac{1}{n}$ and $\frac{1}{r} - \frac{1}{p} <
      \frac{1}{n}$;

    \item[(ii)] If $1 < p, q, r \leq \infty$ satisfy $\frac{1}{p} - \frac{1}{q}
      \leq \frac{1}{n}$ and $\frac{1}{r} - \frac{1}{p} \leq \frac{1}{n}$.
    \end{enumerate}

  \item[(f)] The following homotopy formula holds:
    \begin{displaymath}
      \omega - R_\veps \: \omega = d A_\veps \: \omega + A_\veps\: d \omega.
    \end{displaymath}
  \end{enumerate}
\end{thm}

Recall that a continuous 1-form $\omega$ on $M$ is said to be closed
in the \textsf{Stokes sense} if
  \begin{displaymath}
    \int_{\del D} \omega = 0,
  \end{displaymath}
  for every $C^1$-immersed 2-disk $D$ with piecewise $C^1$
  boundary. It is closed in the \textsf{weak sense} if its weak
  differential is zero, i.e.,
  \begin{displaymath}
    \int_M \omega \wedge d\eta = 0,
  \end{displaymath}
  for every smooth $(n-2)$-form $\eta$.

\begin{lem}    \label{lem:regularization}
  A continuous 1-form $\omega$ on $M$ is closed in the weak sense if
  and only it is closed in the Stokes sense.
  
\end{lem}

\begin{proof}
  $(\Rightarrow)$ Assume $d\omega = 0$ in the weak sense. Fix
  $\veps > 0$. Since $\omega$ is continuous and weakly closed, it
  follows that $\omega \in \Omega^1_{\infty,\infty}(M)$. By
  Theorem~\ref{thm:goldshtein+troyanov} (f) we have:
  \begin{equation}   \label{eq:homotopy}
    \omega - R_\veps \: \omega = d A_\veps \: \omega.
  \end{equation}
  Furthermore, by Theorem~\ref{thm:goldshtein+troyanov} (e), it
  follows that
  $u_\veps := A_\veps \: \omega \in \Omega^0_{\infty,\infty}(M)$,
  i.e., $u_\veps$ is Lipschitz. Thus $du_\veps$ exists a.e. in the
  Fr\'echet sense (and a.e. equals the weak differential of
  $u_\veps$). Moreover, by \eqref{eq:homotopy}
  $du_\veps = \omega - R_\veps \: \omega$, so $du_\veps$ coincides
  a.e. with a continuous 1-form. Thus $u_\veps$ can be chosen to be
  $C^1$. If $D$ is a $C^1$-immersed 2-disk with piecewise $C^1$
  boundary, then:
  \begin{displaymath}
    \int_{\del D} \omega = \int_{\del D} ( R_\veps \: \omega +
    du_\veps ) = 0,
  \end{displaymath}
  since $d R_\veps \: \omega = 0$. Therefore, $\omega$ is closed in
  the Stokes sense. \\

  \noindent $(\Leftarrow)$ Assume now $d\omega = 0$ in the Stokes
  sense. If $U$ is a sufficiently small simply connected set in $M$,
  then on $U$ we have $\omega = dg$, for some $C^1$ function
  $g : U \to \R$. Let $d\eta$ be an arbitrary smooth exact
  $(n-1)$-form. Let $\{ (U_i, \psi_i) \}$ be a smooth partition of
  unity on $M$, where $U_i$ is a sufficiently small disk such that
  there exists a $C^1$-function $g_i$ with $\omega = dg_i$ on
  $U_i$. Then $\eta = \sum_i \eta_i$, where $\eta_i = \psi_i \eta$ is
  supported in $U_i$. It follows that
  \begin{align*}
    \int_M \omega \wedge d\eta & = \sum_i \int_{U_i} \omega \wedge
                                 d\eta_i \\
                               & = \sum_i \int_{U_i} dg_i \wedge d\eta_i \\
                               & = -\sum_i \int_{U_i} d (dg_i \wedge \eta_i) \\
                               & = - \sum_i \int_{\del U_i} dg_i \wedge \eta_i \\
    & = 0,
  \end{align*}
  since $\eta_i = 0$ on $\del U_i$. Therefore, $\omega$ is closed in
  the weak sense.
\end{proof}

\section{A Liv\v{s}ic theorem on the space of differential forms}
\label{sec:livshitz}

In its basic form, the classical Liv\v{s}ic equation over an Anosov
flow is an equation of the form $X\varphi = f$, where $f$ and
$\varphi$ are real-valued functions on $M$. (An analogous
cohomological equation has also been studied over Anosov
diffeomorphisms, partially hyperbolic diffeomorphisms, and other types
of dynamical systems.) In the category of H\"older continuous
functions, the original proof of the existence of solutions was
established in the seminal work of Liv\v{s}ic~\cite{livsic+71,
  livsic+72}. In the smooth case the result was proved by de la Llave,
Marco, and Moriy\'on \cite{L+M+M+86}, and the Sobolev regularity case
was treated in \cite{dll01}. A proof of the classical (as well as the
smooth one, assuming volume-preservation) result using microlocal
analysis was done in \cite{guillarmou+2017}. A Liv\v{s}ic theorem for
sections of vector bundles also using microlocal analysis was recently
established in \cite{cekic+lefeuvre+2025}. See also Lefeuvre's book
\cite{lefeuvre_book+2025} for a more comprehensive (and readable)
survey of results and references.

The main goal of this paper is to investigate the obstacles to the
solvability of the Liv\v{s}ic equation on the space of differential
forms of different degrees using somewhat elementary means (i.e.,
without the use of microlocal analysis).\footnote{However, we do hope
  that in the near future using the heavy machinery of microlocal
  analysis may lead to results stronger than the ones in this paper.}

\paragraph{Invariant forms.} We will first describe the set of
invariant differential forms in all degrees. We set
\begin{displaymath}
  \text{Inv}^k(M,X) = \{ \omega \in C^1_X \Lambda^k(M) : L_X \omega = 0 \}.
\end{displaymath}
Some of the results in the following Proposition are well-known and
elementary, but we include them for completeness.

\begin{prop}   \label{prop:invariant}
  Let $\Phi$ be a smooth Anosov flow with infinitesimal generator
  $X$, preserving a smooth volume form $\Omega$. Then:

  \begin{enumerate}
  \item[(a)] $\text{\normalfont Inv}^0(M,X)$ consists of constant functions.
  \item[(b)] $\text{\normalfont Inv}^1(M,X) = \R \alpha$.
  \item[(c)] If the flow is asymmetric and $2 \leq k \leq n-2$, then
    $\text{\normalfont Inv}^k(M,X) = \{ \mathbf{0} \}$.
  \item[(d)] $\text{\normalfont Inv}^{n-1}(M,X) = \R \: i_X \Omega$.
  \item[(e)] $\text{\normalfont Inv}^n(M,X) = \R \Omega$.
  \end{enumerate}
  
\end{prop}

\begin{proof}
  (a) and (e) are clear. To prove (b), assume $L_X \omega = 0$, for
  some $\omega \in C^1_X \Lambda^1(M)$.  Then
  $f_t^\ast \omega = \omega$, for all $t$, which clearly implies that
  $\omega(v) = 0$, for all $v \in E^{ss} \oplus E^{uu}$. Thus
  $\omega = \psi \alpha$, for some continuous $\psi : M \to \R$. Since
    \begin{displaymath}
      0 = L_X \omega = (X\psi) \alpha = \psi L_X \alpha = (X \psi) \alpha
    \end{displaymath}
    it follows that $\psi$ is flow invariant. Since the flow is
    ergodic (being volume preserving), $\psi$ is constant a.e., hence
    constant by continuity.

    (c) Assume the flow is asymmetric, $2 \leq k \leq n-2$, and
    $L_X \eta = 0$, for a continuous $k$-form $\eta$. We again have
    $f_t^\ast \eta = \eta$, for all $t$. Let $v_1, \ldots, v_k$ be
    arbitrary linearly independent vectors in the same tangent space
    of $M$. We claim that $\eta(v_1,\ldots,v_k) = 0$. Since
    $TM = E^{cs} \oplus E^{uu}$ and $\eta$ is multilinear, it is
    sufficient to show $\eta(v_1,\ldots,v_k) = 0$ in the following two
    cases:

  \begin{enumerate}
  \item[\textbf{Case 1}:] $\{ v_1, \ldots, v_k \} \subset E^{cs}$.
  \item[\textbf{Case 2}:]
    $\{ v_1, \ldots, v_k \} \subset E^{uu} \cup E^{cs}$ and at
    least one vector $v_j$ is in $E^{uu}$.
  \end{enumerate}

  \nopagebreak[4] In Case 1, by decomposing each $v_j$ into the sum
  $v_j = v_j^c + v_j^s \in E^c \oplus E^{ss}$, using the flow
  invariance of $\eta$, and the fact that $k \geq 2$, we obtain
    \begin{displaymath}
      \eta(v_1,\ldots,v_k) = \eta(f_{t \ast}(v_1), \ldots, f_{t
        \ast}(v_k)) \to 0,
    \end{displaymath}
    as $t \to +\infty$.

    In Case 2, the asymmetry of the flow implies 
    \begin{displaymath}
      \eta(v_1,\ldots,v_k) = \eta(f_{t \ast}(v_1), \ldots, f_{t
        \ast}(v_k)) \to 0,
    \end{displaymath}
    as $t \to -\infty$. Thus $\eta = 0$, as desired.

    To prove (d), assume $L_X \Theta = 0$, for some 
    $\Theta \in C^1_X \Lambda^{(n-1)}(M)$. Observe that since
    $L_X (i_X \Theta) = i_X L_X \Theta = 0$\footnote{Note that $i_X$
      and $L_X$ do commute on $C^1_X \Lambda^\ast(M)$.}, $i_X \Theta$
    is a continuous invariant $(n-2)$-form, hence zero by (c).

    Consider the continuous $n$-form $\alpha \wedge \Theta$. Since
    $L_X(\alpha \wedge \Theta) = L_X \alpha \wedge \Theta + \alpha
    \wedge L_X \Theta = 0$, $\alpha \wedge \Theta$ is invariant, hence
    $\alpha \wedge \Theta = c \: \Omega$, for some constant $c$. It
    follows that
    \begin{displaymath}
      \Theta = i_X (\alpha \wedge \Theta) = i_X (c \: \Omega) = c \:
      i_X \Omega,
    \end{displaymath}
    as desired.
\end{proof}

\begin{thm}[Liv\v{s}ic theorem for forms of intermediate degree]    \label{thm:Lie}
  Let $\Phi$ be an asymmetric Anosov flow on a closed manifold $M$,
  and let $\xi$ be a continuous $k$-form on $M$, with
  $2 \leq k \leq n-2$. Then:

  \begin{enumerate}

  \item[(a)] There exists a unique continuous $k$-form $\eta$ such
    that $L_X \eta = \xi$.

  \item[(b)] If $\xi, E^{uu}$, and $E^{cs}$ are $C^1$ (as in the case
    of volume-preserving codimension one Anosov flows in dimensions
    $n \geq 4$), then there exists a family $(\eta_t)_{t \geq 0}$ in
    $C^1 \Lambda^k(M)$ such that
  \begin{displaymath}
    \eta_t \to \eta \quad \text{and} \quad L_X \eta_t \to L_X \eta = \xi,
  \end{displaymath}
  as $t \to \infty$, both with respect to the $C^0$-norm. Each
  $L_X \eta_t$ is also $C^1$.

  \end{enumerate}
\end{thm}

\begin{proof}
  (a) (\textsf{Uniqueness}) Follows directly from Proposition~\ref{prop:invariant}
  (c).

  \medskip
   
  \noindent (\textsf{Existence}) To prove the existence of $\eta$,
  given a continuous $k$-form $\xi$, we need to define
  $\eta(v_1,\ldots,v_k)$ for all vectors $v_1, \ldots,v_k \in TM$. By
  the same argument as in the proof of part (c) of Proposition~\ref{prop:invariant},
  it is enough to specify $\eta(v_1,\ldots,v_k)$ in Cases 1 and 2
  defined above, then extend $\eta$ by multi-linearity and the
  alternating property.

    For $t > 0$ define
    \begin{displaymath}
    \eta_t(v_1,\ldots,v_k) =
    \begin{cases}
      - {\displaystyle \int_0^t (f_s^\star \xi) (v_1,\ldots,v_k)
        \: ds} & \text{in Case 1,} \\
      {\displaystyle \int_0^t (f_{-s}^\star \xi) (v_1,\ldots,v_k)
        \: ds} & \text{in Case 2}.
    \end{cases}
  \end{displaymath}
  The asymmetry of the flow guarantees that $\eta_t$ converges, as
  $t \to \infty$, in the $C^0$-sense to a continuous form $\eta$. It
  is clear that if $\xi, E^{uu}$, and $E^{cs}$ are $C^1$, then so is
  $\eta_t$, for every $t \geq 0$.
   
  Let us show that $L_X \eta = \xi$, i.e.,
  $L_X \eta(v_1,\ldots, v_k) = \xi(v_1,\ldots, v_k)$, for all
  $v_1, \ldots, v_k \in TM$. As above, it suffices to prove this
  in each of the two cases above. In Case 1, we have:
   \begin{align*}
     (f_{\tau}^\ast \eta)(v_1,\ldots,v_k) 
     & = \eta(f_{\tau \ast}(v_1),\ldots, f_{\tau \ast}(v_k)) \\
     & = - \int_0^\infty f_s^\ast \xi(f_{\tau \ast}(v_1),
       \ldots, f_{\tau \ast}(v_k)) \: ds \\
     & = - \int_0^\infty f_{s+\tau}^\ast \xi(v_1,\ldots,v_k) \:
       ds \\
     & = - \int_0^\infty f_t^\ast \xi(v_1,\ldots,v_k) \:
       ds + \int_0^{\tau} f_t^\ast \xi(v_1,\ldots,v_k) \: ds \\
     & = \eta + \int_0^{\tau} f_t^\ast \xi(v_1,\ldots,v_k) \: ds,
   \end{align*}
   for all $\tau \geq 0$. Differentiating both sides with respect ot
   $\tau$ at zero, we obtain
   $L_X \eta(v_1,\ldots,v_{n-2}) = \xi(v_1,\ldots,v_k)$. Case 2 is
   dealt with in a similar way. This proves that $L_X \eta = \xi$.

   A similar calculation yields
   \begin{displaymath}
     L_X \eta_t =
     \begin{cases}
       \xi - f_t^\star \xi & \text{in Case 1}, \\
       \xi - f_{-t}^\star \xi & \text{in Case 2}.
     \end{cases}
   \end{displaymath}
   It follows that $L_X \eta_t \to \xi$, as $t \to \infty$, in the
   $C^0$-sense; in Case 2 this follows again by asymmetry. Finally,
   observe that if $\xi, E^{uu}$, and $E^{cs}$ are $C^1$, then so is $L_X \eta_t$.
 \end{proof}

 \paragraph{The action of $L_X$ in all degrees.} We now investigate
 the action of the Lie derivative $L_X$ on differential forms of all
 degrees. The well-known results are included for completeness.

\begin{thm}     \label{thm:decomposition}
  Let $\Phi$ be a smooth Anosov flow on a closed manifold $M$. 

  \begin{enumerate}
  \item[(a)] If $\Phi$ is transitive, then the image of
    $L_X^{(0)} : C^1_X(M) \to C^0(M)$ consists of continuous functions whose
    integral over all periodic orbits equals zero.

  \item[(b)] If $\Phi$ is transitive, then the image of
    $L_X^{(1)} : C^1_X \Lambda^1(M) \to C^0 \Lambda^1(M)$ consists of continuous
    1-forms $\omega$ such that $\int_\gamma \omega = 0$, for all
    periodic orbits $\gamma$ of $\Phi$.

  \item[(c)] If $\Phi$ is asymmetric and $2 \leq k \leq n-2$, then
    $L_X^{(k)} : C^1_X \Lambda^k(M) \to C^0 \Lambda^k(M)$ is a bijection.

  \item[(d)] If $\Phi$ preserves a smooth volume form $\Omega$, then
    for every $C^1_X$-Riemannian metric $g$ on $M$, we have a
    $g$-orthogonal decomposition:
  \begin{displaymath}
    L^2 \Lambda^{n-1}(M) = \overline{\text{\normalfont image}(L_X^{(n-1)})} \oplus_g
    \mathbb{R} (\star_g \alpha), 
  \end{displaymath}
  where the closure is taken relative to the $L^2$-topology. If $g$ is
  an Anosov metric, then
  \begin{displaymath}
     L^2 \Lambda^{n-1}(M) = \overline{\text{\normalfont image}(L_X^{(n-1)})} \oplus_g
    \mathbb{R} (i_X \Omega).
  \end{displaymath}
  \item[(e)] If $\Phi$ preserves a smooth volume form $\Omega$, then
  the image of $L_X^{(n)} : C^1_X \Lambda^n(M) \to C^0 \Lambda^n(M)$
  consists of all $n$-forms of type $(X\psi) \Omega$, where
  $\psi \in C^1_X(M)$.
    
  \end{enumerate}

\end{thm}

\begin{proof}
  Part (a) is just the classical Liv\v{s}ic theorem. Part (c) is a
  restatement of Theorem~\ref{thm:Lie}. Part (e) is easy to prove. To
  prove (b), let us first show that if $\omega = L_X \xi$, for some
  $\xi \in C^1_X \Lambda^1(M)$, then $\int_\gamma \omega = 0$, for
  every closed orbit $\gamma$. Indeed, for every periodic orbit
  $\gamma$, we have:
  \begin{align*}
    \int_\gamma \omega & = \int_\gamma L_X \xi \\
                       & = \int_\gamma \left. \frac{d}{dt} \right|_0 f_t^\ast \xi \\
                       & = \left. \frac{d}{dt} \right|_0 \int_\gamma f_t^\ast \xi \\
                       & = \left. \frac{d}{dt} \right|_0 \int_\gamma \omega \\
    & = 0.
  \end{align*}
  Now assume that $\int_\gamma \omega = 0$, for every periodic orbit
  $\gamma$. Let $\varphi = \omega(X)$. Since the integral of $\varphi$
  over every periodic orbit $\gamma$ is zero, the classical Liv\v{s}ic
  theorem yields a function $\psi \in C^1_X(M)$ such that
  $\varphi = X\psi$. Set $\beta = \alpha \wedge \omega$. It follows
  from (c) that $\beta = L_X \xi$, for some continuous 2-form
  $\xi$. Contracting $\alpha \wedge \omega = L_X \xi$ by $X$ (i.e.,
  applying $i_X$ to both sides), we obtain
  \begin{displaymath}
    \omega - \varphi \alpha = i_X L_X \xi = L_X (i_X \xi).
  \end{displaymath}
  Thus
  \begin{displaymath}
    \omega = (X \psi) \alpha + L_X (i_X \xi) = L_X (\psi \alpha + i_X \xi),
  \end{displaymath}
  as desired.

  Part (d) follows from Propositions~\ref{prop:general-L_X} and
  \ref{prop:invariant}. Indeed,
  \begin{displaymath}
    L^2 \Lambda^{n-1}(M) = \overline{\text{image}(\normalfont
      L_X^{(n-1)})} \oplus_g \text{image}(\normalfont
      L_X^{(n-1)})^{\perp_g}.
  \end{displaymath}
  By Proposition~\ref{prop:general-L_X},
  $\text{image}(\normalfont L_X^{(n-1)})^{\perp_g} = \ker(\ast_g
  L_X^{(1)} \ast_g)$. Since $\star_g$ is an isomorphism and
  $\ker(L_X^{(1)}) = \R \alpha$ (Prop.~\ref{prop:invariant} (b)), the result follows. Recall that if
  $g$ is an Anosov metric, then
  $\star_g (i_X \Omega) = (-1)^{n-1} \alpha$.
\end{proof}

\begin{remark}
  Observe that if
  $L^2 B^{n-1}(M) \subset \overline{\text{image}(\normalfont
    L_X^{(n-1)})}$, then part (d) of Theorem~\ref{thm:decomposition}
  implies that $i_X \Omega$ is $g$-orthogonal to exact forms, where
  $g$ is any Anosov metric for the flow. Since $i_X \Omega$ is also
  closed, it follows that $i_X \Omega$ is harmonic with respect to $g$
  (at least formally speaking, since $g$ is not smooth). Thus the main
  theorem is consistent with the result of \cite{simic+23}, which
  states that $i_X \Omega$ is intrinsically harmonic if and only the
  flow admits a global cross section (where $X$ is allowed to be any
  non-singular smooth vector field which preserves a smooth volume
  form $\Omega$).
\end{remark}

\begin{cor}
  We have:
  \begin{displaymath}
    \overline{\text{\normalfont image}( L_X^{(n-1)} \! \restrict_{C^1
        Z^{n-1}(M)})} = \overline{\text{\normalfont image}(
      L_X^{(n-1)} \! \restrict_{C^1 B^{n-1}(M)})},
  \end{displaymath}
  where the closures are taken in $L^2 \Lambda^{n-1}(M)$.
\end{cor}

\begin{proof}
  Let $g$ be a smooth Riemannian metric on $M$. It suffices to show
  \begin{displaymath}
    \text{\normalfont image}( L_X^{(n-1)} \! \restrict_{C^1
        Z^{n-1}(M)})^{\perp_g} = \text{\normalfont image}(
      L_X^{(n-1)} \! \restrict_{C^1 B^{n-1}(M)})^{\perp_g}.
    \end{displaymath}
    The $\subset$ part of the proof is clear. Let us show the
    $\supset$ part. Let $\Theta \in \text{\normalfont image}(
      L_X^{(n-1)} \! \restrict_{C^1 B^{n-1}(M)})^{\perp_g}$ and
      $\omega \in C^1 Z^{n-1}(M)$ be arbitrary. We will show that
      $\langle \Theta, L_X \omega \rangle_g = 0$.

      First observe that $L_X \omega = d i_X \omega  + i_X d\omega = d
      i_X \omega$. Next, by Theorem~\ref{thm:decomposition} (d), we
      have
      \begin{displaymath}
        \omega = \lim_{j \to \infty} L_X \xi_j + c \: i_X \Omega,
      \end{displaymath}
      for some $\xi_j \in C^1_X \Lambda^{n-1}(M)$ and
      $c \in \R$. (This decomposition is not orthogonal with
        respect to $g$, but that will not matter.) It follows that
      \begin{displaymath}
        i_X \omega = i_X (\lim_{j \to \infty} L_X \xi_j) = \lim_{j \to
          \infty} i_X L_X \xi_j = \lim_{j \to \infty} i_X d i_X \xi_j.
      \end{displaymath}
\nopagebreak      Thus:
      \begin{align*}
        \langle L_X \omega, \Theta \rangle_g & = \int_M d i_X \omega
                                               \wedge  \star_g \Theta \\
                                             & = \int_M i_X \omega
                                               \wedge d (\star_g \Theta) \\
        & = \lim_{j \to \infty} \int_M i_X d i_X \xi_j \wedge d (\star_g \Theta) \\
        & = \lim_{j \to \infty} \int_M d i_X d i_X \xi_j \wedge
          \star_g \Theta \\
        & = \lim_{j \to \infty} \langle d i_X d i_X \xi_j, \Theta
          \rangle_g \\
        & =  \lim_{j \to \infty} \langle L_X (d i_X \xi_j), \Theta
          \rangle_g \\
        & = 0,
      \end{align*}
      since $\Theta \perp_g \text{\normalfont image}(
      L_X^{(n-1)} \! \restrict_{C^1 B^{n-1}(M)})$. This completes the
      proof.
    \end{proof}

    \begin{remark}
      The Corollary remains true if $C^1$ is replaced by $C^\infty$
      on both sides.
    \end{remark}

    Consider the Lie algebra $\mathfrak{X}(M,\Omega)$ of smooth
    divergence-free vector fields on $M$ (i.e.,
    $X \in \mathfrak{X}(M,\Omega)$ if $X$ is smooth and
    $L_X \Omega = 0$). Denote by $\text{Comm}(X,\Omega)$ its
    \textsf{commutator} subalgebra spanned by the Lie brackets $[Y,Z]$,
    where $Y, Z \in \mathfrak{X}(M,\Omega)$. It is well-known (cf.,
    \cite{arnold+69, lichnerowicz+74}) that
    there is a natural identification of $\mathfrak{X}(M,\Omega)$ with
    closed $(n-1)$-forms on $M$ and of $\text{Comm}(X,\Omega)$ with
    exact $(n-1)$-forms via the map $Z \mapsto i_Z \Omega$.

    \begin{cor}
      For every $Z \in \mathfrak{X}(M,\Omega)$ there is a sequence
      $(W_j)$ in $\text{\normalfont Comm}(X,\Omega)$ such that
      \begin{displaymath}
        [X,W_j] \to [X,Z],
      \end{displaymath}
      as $j \to \infty$, in the $L^2$-sense.
    \end{cor}

    \begin{proof}
      Let $Z \in \mathfrak{X}(M,\Omega)$ be arbitrary. Since
      \begin{displaymath}
        i_{[X,Z]}  \Omega = d i_X i_Z \Omega = L_X (i_Z \Omega)
      \end{displaymath}
      the previous Corollary yields a sequence $(d \xi_j)$ in $C^\infty
      B^{n-1}(M)$ such that $L_X (d\xi_j) \to L_X (i_Z \Omega)$, as $j
      \to \infty$, in the $L^2$-sense. Since $C^\infty B^{n-1}(M)$
      corresponds to $\text{Comm}(X,\Omega)$ via the map $W \mapsto
      i_W \Omega$, there exists a sequence $(W_j)$ in
      $\text{Comm}(X,\Omega)$ such that $d\xi_j = i_{W_j} \Omega$. Thus:
      \begin{align*}
        i_{[X,Z]} \Omega & = L_X (i_Z \Omega) \\
                         & = \lim_{j \to \infty} L_X (d\xi_j) \\
                         & = \lim_{j \to \infty} L_X (i_{W_j} \Omega) \\
                         & = \lim_{j \to \infty} d i_X i_{W_j} \Omega \\
                         & = \lim_{j \to \infty} i_{[X,W_j]} \Omega_j.
      \end{align*}
      It follows that $[X,W_j] \to [X,Z]$, as $j \to \infty$, as
      desired.
    \end{proof}

\section{Proof of the Main Theorem}
\label{sec:proof}

$(\Rightarrow)$ Assume
\begin{displaymath}
  L^2 B^{n-1}(M) \subset
  \overline{\text{image}(L_X)}.
\end{displaymath}
Let $\omega \in C^\infty \Lambda^{n-2}(M)$ be arbitrary. Since
$d\omega \in C^\infty B^{n-1}(M) \subset L^2 B^{n-1}(M)$, there
exists a sequence
$(\Theta_j)$ in $C^1_X \Lambda^{n-1}(M)$ such that
\begin{displaymath}
  L_X \Theta_j \to d\omega,
\end{displaymath}
as $j \to \infty$, in the $L^2$-sense. Let $g$ be an arbitrary smooth
Riemannian metric on $M$. Then:
\begin{align*}
  \int_M d\omega \wedge \alpha & = (-1)^{n-1} \langle d\omega, \star_g
                                 \alpha \rangle_g \\
  & = (-1)^{n-1} \lim_{j \to \infty}  \langle L_X \Theta_j, \star_g \alpha
    \rangle_g \\
                               & = \lim_{j \to \infty} \int_M L_X
                                 \Theta_j \wedge \alpha \\
                               & = 0,
\end{align*}
by integration by parts, since $L_X \alpha = 0$. Thus $\alpha$ is
weakly closed, hence closed in the Stokes sense, by
Lemma~\ref{lem:regularization}. By the Hartman-Frobenius theorem, it
follows that $E^{ss} \oplus E^{uu}$ is uniquely integrable, which, by
\cite{plante72} implies that the flow is topologically conjugate to
the suspension of an Anosov diffeomorphism.

\medskip

\noindent $(\Leftarrow)$ Assume now that $E^{ss} \oplus E^{uu}$ is
uniquely integrable. By the Hartman-Frobenius theorem, $d\alpha$
exists in the Stokes sense and is continuous. Since it is also
invariant, it follows without difficulty that $d\alpha = 0$, also in
the Stokes sense. By Lemma~\ref{lem:regularization}, $d \alpha = 0$
also in the weak sense.

Back to the proof of the Main Theorem, assume that $\Theta = d\xi$ is
a smooth \emph{exact} $(n-1)$-form. Then by
Theorem~\ref{thm:decomposition} we can write
\begin{displaymath}
  \Theta = \hat{\Theta} + c \: i_X \Omega,
\end{displaymath}
for some $\hat{\Theta} \in \overline{\text{image}(L_X)}$ and a
constant $c$. It is enough to show $c = 0$.

Since $\hat{\Theta} \in \overline{\text{image}(L_X)}$, we have
$\Theta = \lim_{j \to \infty} L_X \Theta_j$, for some sequence of smooth
$(n-1)$-forms $(\Theta_j)$ (the limit being in the $L^2$-sense). Observe
that
\begin{displaymath}
  \int_M \alpha \wedge \Theta = \int_M \alpha \wedge d\xi = 0,
\end{displaymath}
by integration by parts and the fact that $\alpha$ is weakly closed.

On the other hand,
\begin{displaymath}
  \int_M \alpha \wedge \Theta = \lim_{j \to \infty} \int_M \alpha
  \wedge L_X \Theta_j + c \int_M \alpha \wedge i_X \Omega = c.
\end{displaymath}
Thus $c = 0$, which implies
$\Theta = \hat{\Theta} \in \overline{\text{image}(L_X)}$. Since
$C^\infty B^{n-1}(M)$ is dense in $L^2 B^{n-1}(M)$ (cf.,
\cite{goldshtein+troyanov+06}), the desired conclusion follows. \qed

\paragraph{Acknowledgement.} I would like to thank the
Department of Mathematics at UC Berkeley for their hospitality during
my sabbatical leave, when this paper was written.

%%%%%%% old reference style %%%%%%%

\bibliographystyle{amsalpha}
\bibliography{main.bib}

\end{document}